\documentclass[12pt]{tglat2e}
\usepackage{amssymb,latexsym,theorem}
\usepackage{hyperref}
\usepackage[all]{xy}
\theoremstyle{plain}
\newtheorem{theorem}{Theorem}
\newtheorem{Lem}[theorem]{Lemma}
\newtheorem{Pro}[theorem]{Proposition}
\newtheorem{Cor}[theorem]{Corollary}
\theoremstyle{definition}
\newtheorem{Def}{Definition}
\newtheorem*{Property}{Property}

\theoremstyle{remark}
\newtheorem{remark}{Remark}

\newcommand{\Spec}{\mathop{\mathrm {Spec}}}

\newcommand{\Z}{{\mathbb Z}}
\newcommand{\Q}{{\mathbb Q}}
\newcommand{\kk}{{\mathbf k}}

\newcommand{\hull}{\mathop{{\mathrm{hull}}}\nolimits}

\newcommand{\gr}{\mathop{{\mathrm{gr}}}}

\newcommand{\ind}{\mathop{{\mathrm{ind}}}\nolimits}

\newcommand{\red}{{\mathrm{red}}}
\newcommand{\opp}{{\mathrm{opp}}}

\newcommand{\h}{{\mathrm{H}}}
\title[An integrality theorem of Grosshans]{An integrality theorem of Grosshans over arbitrary base ring}
\authors{Wilberd van der Kallen
\address Mathematisch Instituut\\ P.O. Box 80.010\\ 3508 TA Utrecht\\ The Netherlands
\email W.vanderKallen@uu.nl}

\dedicatory{Dedicated to the memory of T.~A.~Springer}
\begin{document}%\sloppy

\maketitle

\begin{abstract}
We revisit a theorem of Grosshans and show that it holds over arbitrary commutative base ring $\kk$.
One considers a split reductive group scheme $G$ acting on a $\kk$-algebra $A$ and leaving invariant a subalgebra $R$.
Let $U$ be the unipotent radical of a split Borel subgroup scheme.
If $R^U=A^U$ then the conclusion is that $A$ is integral over $R$.
\end{abstract}

\section*{Introduction}
In \cite{Grosshans contr} Grosshans considered a reductive algebraic group $G$ defined over an algebraically closed field $\kk$
acting algebraically on a commutative $\kk$-algebra $A$. Fix a Borel subgroup $B$ with unipotent radical $U$.
Then Grosshans considered the smallest $G$-invariant $\kk$-subalgebra $G\cdot A^U$ of $A$ that contains the fixed point
algebra $A^U$.
He showed that $A$ is integral over $G\cdot A^U$. If $R$ is any other $G$-invariant $\kk$-subalgebra of $A$ that contains $A^U$ 
it then follows that $A$ is integral over $R$. One of the tools used by Grosshans is what is called power reductivity
in \cite{FvdK}. As it is shown in \cite{FvdK} that power reductivity holds over arbitrary commutative base ring $\kk$, we now 
set out to prove the integrality result of Grosshans in the same generality. We need a little care as we are not even assuming
that the ground ring is noetherian.

\section{Preliminaries}
We use an arbitrary commutative ring $\kk$ as base ring. 
Let $A$ be a commutative $\kk$-algebra.
We say that an affine algebraic group scheme $G$ acts on $A$  if $A$ is a $G$-module \cite{J} and
 the multiplication map $A\otimes_\kk A\to A$ is a $G$-module map. Then the coaction $A\to A\otimes_\kk \kk[G]$ is 
an algebra homomorphism. One also says that $G$ acts rationally on $A$ 
by algebra automorphisms. Geometrically it means that $G$ acts from the right on $\Spec A$.

\begin{Lem}\label{nilinv}
 Let $G$ be a smooth affine algebraic group scheme over $\kk$. 
Let $G$ act on the commutative $\kk$-algebra $A$. Then the nilradical of $A$ is a $G$-submodule.
\end{Lem}

\begin{proof}
(Thanks to Angelo Vistoli \url{http://mathoverflow.net/questions/68366/}
for explaining to me  
  that smoothness is the right condition.) 

As the base change map
$A_\red\to A_\red\otimes_\kk \kk[G]$ is a smooth map, $A_\red\otimes_\kk \kk[G]$ is reduced,
by \cite[Prop. (17.5.7)]{EGA IV} or by \cite[Lemma 033B]{stacks} with URL
\url{http://stacks.math.columbia.edu/tag/033B}.

Now let $N$ denote the nilradical of $A$. 
The coaction $A\to A\otimes \kk[G]$ sends $N$ to the nilradical $N\otimes_\kk \kk[G]$ of
$A\otimes_\kk \kk[G]$.  \qed
\end{proof}

From now on let $G=G_\kk$, where $G_\Z$ is a Chevalley group over $\Z$. 
In other words, $G$ is a split reductive group scheme over 
$\kk$ under the conventions of \cite{SGA3}.
Choose %a pinning, known in French as an
%\emph{\'epinglage} \cite[Expos\'e XXIII]{SGA3}. We get   
a split maximal torus $T$, a standard Borel
subgroup $B$ and its unipotent radical $U$. 

\begin{Lem}
 The coordinate ring $\kk[G]$ is a free $\kk$-module.
\end{Lem}

\begin{proof}
 As $\kk[G]=\Z[G_\Z]\otimes_\Z \kk$ it suffices to treat the case $\kk=\Z$. Now the coordinate ring of $G$ is a subring of the 
coordinate ring of the big cell. And the coordinate ring of the big cell is clearly free as a $\Z$-module.
Now use that a submodule of a free $\Z$-module is free \cite[Chapter I, Theorem 5.1]{HS}. \qed
\end{proof}

\begin{Lem}\label{genby}
 If $V$ is a  $G$-module and $v\in V$, then the $G$-submodule generated by $v$ exists and is finitely
generated as a $\kk$-module.
\end{Lem}

\begin{proof}
 As $\kk[G]$ is a free $\kk$-module, this follows from \cite[Expos\'e VI, Lemme 11.8]{SGA3}.  \qed
\end{proof}

See also \cite[Proposition 3]{Se}.
Note that the existence result in the Lemma does not follow from the fact that $G$ is flat over $\kk$ 
\cite[Expos\'e VI, \'Edition 2011, Remarque 11.10.1]{SGA3}.

\begin{Def}
 Recall that we call a homomorphism of  $\kk$-algebras $f:A\to B$ 
\emph{power surjective} \cite[Definition 2.1]{FvdK}
if for every $b\in B$ there is an $n\geq1$ so that the power $b^n$ is in the image
of $f$.

A flat affine group scheme $H$ over $\kk$ is called \emph{power reductive} \cite[Definition 2]{FvdK}
if the following holds.
\begin{Property}[Power Reductivity]
Let $L$ be a cyclic $\kk$-module with trivial $H$-action. 
Let $M$ be a rational $H$-module, and let $\varphi$ be an $H$-module map from $M$ onto $L$.
Then there is a positive integer $d$ such that the $d$-th symmetric power of $\varphi$ induces a surjection:
\[
(S^d M)^H\to S^d L .
\]
Here $V^H=H^0(H,V)$ denotes the submodule of invariants in an $H$-module $V$.
\end{Property}
\end{Def}

\begin{Pro}\label{old new}
 Let $H$ be a flat  affine algebraic group scheme over $\kk$.
The following are equivalent
\begin{enumerate}
 \item $H$ is power reductive,\label{old def}
\item for every power surjective
$H$-homomorphism of  
commutative $\kk$-algebras $f:A\to B$ the map $A^H\to B^H$ is power surjective.\label{new def}
\end{enumerate}

\end{Pro}
\begin{proof}
First assume \ref{old def}.
Let $f:A\to B$ be power surjective and let $b\in B^H$. Choose $n\geq1$ so that $b^n\in f(A)$. 
Let $M=f^{-1}(L)$ be the inverse image of $L=\kk b^n$. Choose $d\geq1$ so that $(S^d M)^H\to S^d L$ is surjective.
Multiplication induces $H$-module maps $S^dA\to A$, $S^dB\to B$.
One has a commutative diagram of $H$-homomorphisms 
\[\xymatrix{
\ar[r]\ar@{->>}[d](S^dM)^H&\ar[r]\ar[d] (S^dA)^H&\ar[d] A^H\ \\
\ar[r]S^dL&\ar[r] (S^dB)^H&B^H,
}\]
and one sees that $b^{nd}$ lies in the image of $A^H$ because it lies in the image of $S^dL$.

Conversely, assume \ref{new def} and let $M\to L$ be given as in the Property. One has a surjective map of symmetric algebras 
$S^*(M)\to S^*(L)$. Now let $b$ be a generator of $S^1L$. There is a power $b^d\in S^dL$ of $b$ that lies in the image of
$(S^*(M))^H$. But then $(S^d M)^H\to S^d L$ is surjective. 
 \qed
\end{proof}

\begin{remark}
 So the finite generation hypothesis does not belong in \cite[Proposition 6]{FvdK}. Note that there is no finiteness hypothesis
on $M$ in the Power Reductivity Property.
\end{remark}

\begin{Pro}\label{Upow}
 Let $G$ act rationally by $\kk$-algebra automorphisms on the commutative 
algebra $A$ and let $J$ be a $G$-invariant ideal. Then 
$A^U\to (A/J)^U$ is power surjective.
\end{Pro}
\begin{proof}
The transfer principle \cite[Ch. Two]{G} tells that $A^U=(A\otimes_\kk \kk[G/U])^G$, 
where $\kk[G/U]$ means the algebra of $U$-invariants
in $\kk[G]$ under the action by right translation. Here is one proof. % Why not compute right hand side as
% $(A\otimes_\kk \kk[G])^{G\times U}$ ? Is much more general.
Write $$A^U=\hom_U(\kk,A)=\hom_G(\kk,\ind_U^GA)=(\ind_B^G\ind_U^BA)^G=(\ind_B^G(A\otimes_\kk \kk[T]))^G,$$
% That $\ind_U^BA$ equals $A\otimes_\kk \kk[T])$ is another example of the tensor identity.
% The proof in \cite[I 3.6]{J} goes through because $\ind_U^B\kk$ is a direct summand of the $\kk$-module $\kk[B]$.
% But one may also argue that  $\hom_B(V,A\otimes_\kk \kk[T])=(\hom_\kk(V,A\otimes_\kk \kk[T])^U)^T=
% (\hom_U(V,A)\otimes_\kk \kk[T])^T=\hom_U(V,A)$.
where the $B$-module $\kk[T]$ is a direct sum of weights of $B$, so that $\ind_B^G(A\otimes_\kk \kk[T])$
 equals $A\otimes_\kk\ind_B^G( \kk[T])$ by the tensor identity for weights \cite[Proposition 17]{FvdK}. 
Further $\ind_B^G \kk[T]=\ind_U^G \kk= \kk[G/U]$, so 
$A^U=(A\otimes_\kk \kk[G/U])^G$ indeed.
We may identify $A^U\to (A/J)^U$ with $(A\otimes_\kk \kk[G/U])^G\to (A/J\otimes_\kk \kk[G/U])^G$.
Now use that $G$ is power reductive \cite[Theorem 12]{FvdK} and apply Proposition \ref{old new}. \qed % 
\end{proof}

\section{The integrality theorem}
If $G$ acts rationally by $\kk$-algebra automorphisms on our algebra $A$, we denote by $G\cdot A^U$ the 
$\kk$-subalgebra generated 
by the $G$-submodules generated by the elements of $A^U$.
Thus $G\cdot A^U$ is the smallest $G$-invariant subalgebra of $A$ that contains $A^U$.
Our main result is the following generalization of {\cite[Theorem 5]{Grosshans contr}}.

\begin{theorem}\label{hullint}
The algebra $A$ is integral over $G\cdot A^U$.
\end{theorem}

\begin{proof}
 As Grosshans works over an algebraically closed field there are some details that need to be checked now.
Let $v\in A$. We have to show that $v$ is integral over $G\cdot A^U$. Let $V$ be the $G$-submodule of $A$ 
generated by $v$ and consider the symmetric algebra $S_\kk^*(V)$ on
$V$. Using the obvious map from $S_\kk^*(V)$ to $A$ one sees that it suffices to prove
the theorem for the algebra $S_\kk^*(V)$. So from now on let $A=S_\kk^*(V)$.  Let $A^+$ be the augmentation ideal
generated by $V$ in $A$. Let $J$ be the ideal of $A$ generated by $A^+\cap (G\cdot A^U)$. 
It is $G$-invariant, so its radical $\sqrt J$ is also $G$-invariant, by lemma \ref{nilinv} applied to $A/J$.
%Let $\sqrtp J$ denote the intersection of $A^+$ with $\sqrt J$.
We claim that $A^+\subseteq \sqrt J$. Suppose not. Then $A^+/(A^+\cap\sqrt J)$ is nontrivial. Now every nontrivial $G$-module
$N$ has at least one nontrivial $U$-invariant.
(Note that by lemma \ref{genby} we may reduce to the case that $N$ has finitely many weight spaces.)
Say $0\neq f\in (A^+/(A^+\cap\sqrt J))^U$. View $f$ as a nonzero element of $A/\sqrt J$.
By Proposition \ref{Upow} there is a power of $f$ that lies in the image of $A^U$ in the algebra $A/\sqrt J$. 
But then it actually lies in the image of $A^+\cap (G\cdot A^U)$, hence in the image of $J$, which is zero. 
But $A/\sqrt J$ is reduced; contradiction.

Let $v_1$, \ldots, $v_n$ generate $V$ as a 
$\kk$-module. Every element $f$ of $J$ may be written as a sum of terms $a_Iv^I$, where $I=(I_1,\ldots,I_n)$, 
$a_I\in A^+\cap (G\cdot A^U)$ and $v^I:=v_1^{I_1}\cdots v_n^{I_n}$. Moreover, if $f$ is homogeneous of degree $d$,
then the $a_I$ may be taken homogeneous of degree $d-|I|$ where $|I|=I_1+\cdots +I_n$. In particular, all terms
have $|I|<d$.
As $A^+\subseteq \sqrt J$ we may choose $m$ so large that $v_i^m\in J$ for all $i$. 
So then $v_i^m$ may be written as a sum of terms
$a_Iv^I$ with $|I|<m$. It follows that $A^+$ is generated as a $G\cdot A^U$-module by finitely many $v^I$. 
The theorem follows.  \qed
\end{proof}

Let $B^\opp$ denote the Borel subgroup scheme containing  $T$ that is opposite to $B$. 
If $V$ is a $T$-module, then $\inf_T^{B^\opp}V$ denotes the
$B^\opp$-module obtained by composition with the standard homomorphism $B^\opp\to T$.
Recall that Grosshans has introduced a filtration on any $G$-module $M$ (after Luna). 
Its associated graded module $\gr M$ can be embedded
into the module $\hull_\nabla(\gr M):=\ind_{B^\opp}^G\inf_T^{B^\opp}M^U $. 
One knows that $(\hull_\nabla(\gr M))^U=(\gr M)^U$ and 
that $\h^i(G,\hull_\nabla(\gr M))$ vanishes for positive $i$. If $G$ acts on the commutative $\kk$-algebra $A$, then  
$\gr A$ and $\hull_\nabla(\gr A)$ are commutative $\kk$-algebras. See \cite{FvdK} for details on all this.

We now get a proof  of \cite[Theorem 32]{FvdK} in the style of Grosshans \cite{Grosshans contr}.

\begin{Cor}
 \label{torsion}
Let $A$ be a finitely generated commutative $\kk$-algebra on which $G$ acts rationally by $\kk$-algebra 
automorphisms. If $\kk$ is Noetherian,
there is a positive integer $n$ so that:
\[
n\;\hull_\nabla(\gr A)\subseteq \gr A
.\]
In particular $\h^i(G,\gr A)$ is annihilated by $n$ for positive $i$.
\end{Cor}

\begin{proof}
 As in the proof of \cite[Theorem 8]{Grosshans contr} theorem \ref{hullint} shows
that $\hull_\nabla(\gr A)$ is integral over $\gr A$. As it is also a finitely generated $\kk$-algebra \cite[Theorem 30]{FvdK},
it is a finitely generated module over $\gr A$. View $\hull_\nabla(\gr A)\otimes_\Z\Q$ as a $G_\Q$-module
 \cite[Remark 52]{FvdK}.
It is a direct sum of modules $\ind_{B^\opp}^G\inf_T^{B^\opp} A^U_\lambda \otimes_\Z\Q$ with highest weight $\lambda$
(if we consider the roots of $B$ positive).
As the image of $\gr A\otimes_\Z\Q$ in $\hull_\nabla(\gr A)\otimes_\Z\Q$ contains the highest weight spaces,
the injection $\gr A\to \hull_\nabla(\gr A)$ becomes an isomorphism
after tensoring with $\Q$. So $\hull_\nabla(\gr A)/\gr A$ is a finitely generated $\gr A$-module and a
torsion abelian group. Choose $n>0$ so that $n$ annihilates $\hull_\nabla(\gr A)/\gr A$.
Then it also annihilates $\h^{i-1}(G,\hull_\nabla(\gr A)/\gr A)$, hence $\h^i(G,\gr A)$,
for $i>0$.  \qed 
\end{proof}

\end{document}